\documentclass[letterpaper, 10 pt, conference]{ieeeconf}
\IEEEoverridecommandlockouts                              % This command is only
\usepackage{blindtext, graphicx}

\usepackage{amssymb}
\usepackage{balance}
\usepackage{amsmath}
\usepackage{graphicx,color}
\usepackage{epsfig}
\usepackage{psfrag}
\usepackage{booktabs} % For formal tables
\usepackage{enumerate}
\usepackage{algpseudocode,algorithm,algorithmicx}
\usepackage{graphicx}
\usepackage{caption}
\usepackage{subcaption}

\algrenewcommand\algorithmicrequire{\textbf{Precondition:}}
\algrenewcommand\algorithmicensure{\textbf{Postcondition:}}

\newtheorem{assum}{Assumption}
\newtheorem{remark}{Remark}
\newtheorem{definition}{Definition}
\newtheorem{proposition}{Proposition}
\newtheorem{example}{Example}

\newtheorem{theorem}{Theorem}
\newtheorem{lemma}{Lemma}

\newtheorem{corollary}{Corollary}[theorem]

\title{\LARGE \bf %A Probabilistic approach to linear systems subject to packet dropouts}
 Probabilistic state transfer, estimation and measures for optimal actuator/sensor placement for linear systems with packet dropouts}
\author{A. Sanand Amita Dilip% <-this % stops a space
\thanks{A. Sanand Amita Dilip is with the Faculty of Electrical Engineering, IIT Kharagpur, India.
        {\tt\small sanand@ee.iitkgp.ac.in}}%
}
\begin{document}
\maketitle
\thispagestyle{empty}
\pagestyle{empty}
%\title{ On the probabilistic controllability and observability of systems with packet dropouts} 

% \author{\IEEEauthorblockN{A. Sanand Amita Dilip}%\authornote{A. Sanand Amita Dilip is supported by ARC}
% \IEEEauthorblockA{UC Louvain, ICTEAM, Belgium.\\
% Email: sanand.athalye@uclouvain.be}%\IEEEauthorrefmark{A. Sanand Amita Dilip is supported by ARC}%\thanks{A. Sanand Amita Dilip is supported by ARC}
% \and
% \IEEEauthorblockN{Nikolaos Athanasopoulos}
% \IEEEauthorblockA{Queen's University Belfast, EEECS, \\ Northern Ireland, UK.\\
% Email: n.athanasopoulos@qub.ac.uk}
% \and
% \IEEEauthorblockN{Rapha\"el M. Jungers}
% \IEEEauthorblockA{UC Louvain, ICTEAM, Belgium.\\
% Email: raphael.jungers@uclouvain.be}\thanks{Rapha\"el M. Jungers  is a Fulbright Fellow and a FNRS Research Associate.}}\thanks{A. Sanand Amita Dilip is supported by ARC}

%%%%%%%%%%%%%%%%%%%%%%%%%%%%%%%%%%%%%%%%%%%%%%%%%%%%%%%%%%%%%%%%%%%%%%%%%%%%%%%%
\begin{abstract}
% We generalize the idea of the controllability/observability Gramian for discrete linear systems to more general class of systems where we allow packet dropouts with 
% a given probability. 
We consider linear systems subject to packet dropouts and obtain 
necessary and sufficient conditions for an arbitrary state transfer and state estimation over a finite time instance $T$. 
The data loss signal is modeled using the Bernoulli random variable. 
We leverage properties of 
the Hadamard product in our approach and use the derived necessary and sufficient conditions to compute the probability that an arbitrary state transfer is possible at a specified 
time instant. Similarly, the probability of finding an exact state estimate is found 
using the observability counterparts of the results. 
%We also introduce a notion of the average energy.  %which 
 %allows us to have a 
 Using the necessary and sufficient conditions obtained for the invertibility of the Gramian, we give  
 new probabilistic measures for optimal actuator and sensor placement problems and obtain optimal/sub-optimal solutions. 
 We demonstrate by an example how the probabilities of packet dropouts influence the choice of an optimal actuator. 
 We also discuss how to implement feedback laws and the LQR problem for these models 
 involving packet dropouts. 
%We quantify the effect of probabilistic packet dropout on controllability/observability of linear systems.
 \end{abstract}
%\IEEEpeerreviewmaketitle
 
 \section{Introduction}
 In many modern control systems, the plant and the controllers are geographically distributed and connected to each other 
 via a communication network. One expects that there are disruptions in  this communication network due to the presence of non-idealities such as packet losses in wireless communication. 
%  real time situations, many times the control loop is disrupted by undesired events such as packet loss in wireless communication. 
 There could be time instances where no actuator input is available for control or no sensor output to observe when there are packet dropouts. This greatly 
 influences system theoretic properties of control systems. 
 We refer the reader to \cite{JungersKunduHeemels},\cite{mishchat},\cite{gom}-\cite{tab},\cite{packetdropsiccps}
 %\cite{paj}, \cite{schenato}, \cite{sin}, \cite{smith}, \cite{tab}, 
  for details on systems with wireless control, 
 packet dropouts and control over lossy networks.
 Discrete-time linear time invariant (LTI) systems are of the form {\small\begin{eqnarray}
  x(t+1)=Ax(t)+Bu(t),\label{dlin}
 \end{eqnarray}}$A\in \mathbb{C}^{n\times n}$, $B \in \mathbb{C}^{n\times m}$. 
%Inhere, we consider time instants when there is a loss of communication between the plant and the controller.  
 Following \cite{JungersKunduHeemels}, 
we express the system $(\ref{dlin})$ subject to packet dropouts as a switching system of the form $x(t+1)=Ax(t)+B\sigma(t)u(t)$, where 
{\small \begin{eqnarray}\label{switch}
 x(t+1) &=&\begin{cases} Ax(t)+Bu(t), \mbox{ if } \sigma(t)=1,\\
 Ax(t), \;\;\;\;\;\;\;\;\;\;\;\;\mbox{ if } \sigma(t)=0 \end{cases}
\end{eqnarray}}and $\sigma(\cdot): \mathbb{N}\rightarrow \{ 0,1\}$ is a binary switching signal taking random values either $0$ or $1$. In other words, we model 
$\sigma$ by a Bernoulli random variable where $p$ is the probability that no packet dropout occurs at a given time instant. 
 
 In \cite{JungersKunduHeemels}, %. Therein,  
  it is shown that there exists an algorithm deciding controllability and observability of \eqref{switch} in finite time when $\sigma$ 
  is subject to constraints defined by a directed graph. Instead of the constrained switching model used therein, we consider a probabilistic model for 
  the communication signal $\sigma$. 
  We give necessary and sufficient conditions for an arbitrary state 
  transfer of $(\ref{switch})$ using 
  the controllability Gramian; which allows us 
  to give a probabilistic measure of the energy required for a state transfer of $(\ref{switch})$ given the probability $p$ of a successful transmission of the input.
 
 The controllability Gramian plays an important role in linear systems. Various metrics on controllability using the controllability Gramian were 
 studied in \cite{Summers}-\cite{shchat} for example, the determinant and/or the trace of the controllability Gramian, the minimum eigenvalue of the controllability
 Gramian, the trace of the inverse of the controllability Gramian etc. 
 In \cite{Pasqualetti}, the problem of controlling complex networks was studied by designing a control input 
 to steer a network to a target state. The minimum eigenvalue of the controllability Gramian was used as a metric to quantify the difficulty of the control 
 problem. We refer the reader to \cite{tzoum} and \cite{fitch}-\cite{muller}  for more details on optimal actuator and sensor placement problems. 
 In this article, we propose a new probabilistic measure using the controllability and the observability Gramian for optimal actuator and sensor placement 
 problems. 
 
 The paper is organized as follows. 
 We study properties of the controllability Gramian associated with linear systems subject to packet dropouts $(\ref{switch})$. 
 %Such systems can be modeled as some specific types of switched systems. %We obtain a new expression of the controllability Gramian using the Hadamard product. 
 Using the Hadamard decomposition of the controllability/observability Gramian, we study the state transfer/estimation
  problems for the proposed models, finding the corresponding probabilities. %and allows us to apply probabilistic properties such as average energy for these models. 
  We obtain necessary and sufficient conditions on the switching signal $\sigma$ (in terms of its non zero entries) 
 for a state transfer/estimation when $A$ is diagonalizable. Then, we propose a new probabilistic measure for optimal actuator/sensor placement problem; 
 using the controllability/observability Gramian associated with a certain class of signals 
 %along with their corresponding probability of occurrence. %of the corresponding dropout signals; and 
 and leverage the obtained results to tackle 
 the optimization problem. Finally, we discuss about the LQR problem and feedback laws for these models and mention a few future work possibilities in conclusion. 
\section{preliminaries}
In this section, we build some preliminaries to be used in the sequel. 
\begin{definition}\label{switching_sig}
 \begin{figure}[hbt]\label{markovmod} 
\centering
  \includegraphics[width=40mm]{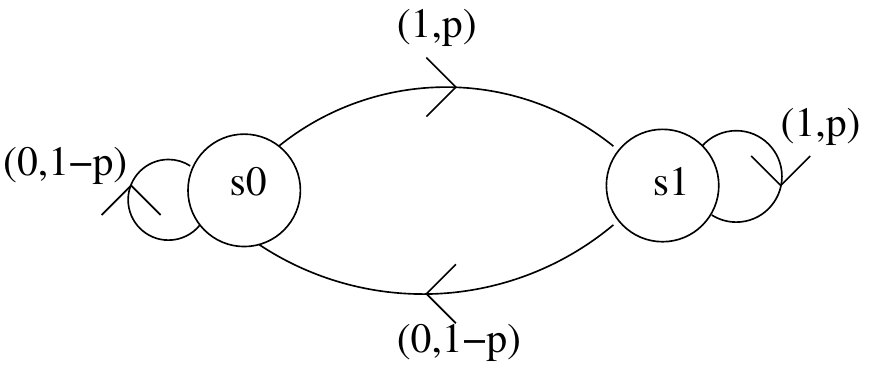}
  \caption{Markovian modeling of a switching signal.}
  \label{3dropnw}
\end{figure}
Suppose that $\forall t \in \mathbb{N}\cup \{0\}$, the probability that no packet dropout occurs at $t$ is $p$. %and the probability that a packet dropout occurs at $t$ is $1-p$. 
 An admissible switching signal $\sigma$ is defined by the Markovian model (Figure $1$) where the two nodes are labeled as $s_0$ and $s_1$ and edges 
 are labeled by pairs $(1,p)$ and $(0,1-p)$.   
  \emph{A sequence $\sigma(0)\sigma(1)\ldots$ is admissible if there exists a path in the graph above 
  where the successive first component of the edge labels carries the sequence.} 
 The probability of occurrence of $\sigma$ is obtained by multiplying the second components of all the edge labels in the path above. The set of all 
 admissible switching signals for length $t$ is denoted by $\mathcal{S}^t$. 
\end{definition}
We now define the controllability matrix for $(\ref{switch})$. 
% Let {\small$C(A,B)$} be the usual controllability matrix
% {\small\begin{eqnarray}
%   C(A,B) := \left[ \begin{array}{cccc} B & AB &\cdots  & A^{n-1}B\end{array}\right].        
% \end{eqnarray}}
%We now give an expression for the controllability matrix for $(\ref{switch})$. 
Rewriting $(\ref{switch})$ as  
{\small\begin{eqnarray}
 x(t)&=&Ax(t-1)+B\sigma(t-1)u(t-1) \\
 &=&A^tx(0)+\sum_{i=0}^{t-1}A^{t-1-i}B\sigma(i)u(i) \\
&=&A^tx(0)+C_{\sigma(t-1)}(A,B)\bar{u}\label{ip}
\end{eqnarray}}
where {\small$C_{\sigma(t-1)}(A,B)=$} 
{\small\begin{eqnarray}\label{controllb}
 %&&\nonumber\\
 &&\left[ \begin{array}{cccc} A^{t-1}B\sigma(0)  &\cdots  & AB\sigma(t-2)& B\sigma(t-1)\end{array}\right]
\end{eqnarray}}is the controllability matrix associated with a signal $\sigma$ at time $(t-1)$ and
{\small\begin{eqnarray}
\bar{u}=\left[ \begin{array}{ccccc}  u^{*}(0)&u^{*}(1)&\cdots&  u^{*}(t-2) & u^{*}(t-1)\end{array}\right]^{*}. 
\end{eqnarray}}
We now give an expression for the controllability 
Gramian for the system $(\ref{switch})$ for a fixed signal $\sigma$ and a fixed time $t$. 
\begin{definition} 
 {\small\begin{eqnarray}
  W_{\sigma(t-1)}:=C_{\sigma(t-1)}(A,B)C_{\sigma(t-1)}(A,B)^{*}
 \end{eqnarray}}is the controllability Gramian associated with system (\ref{switch}) at time $(t-1)$ with respect to the switching signal $\sigma$. %\QEDopen
\end{definition}
Recall that for classical discrete LTI systems, the controllability Gramian is given by {\small$W_{t}(A,B)=\sum_{i=0}^{t}A^iBB^{*}(A^{*})^i$}. 
For a fixed time $t$ and a fixed signal $\sigma$, the controllability Gramian for $(\ref{switch})$ is also given by 
{\small\begin{eqnarray}
 W_{\sigma(t)}= \sum_{i=0}^{t}\sigma(t-i)A^iBB^{*}(A^{*})^i.\label{gramexpsum}
\end{eqnarray}}
In the following proposition, we state how to compute the energy required for a state transfer from $x(0)$ to $x(t)$ 
for a fixed signal $\sigma$ using {\small$W_{\sigma(t)}$}. 
\begin{proposition}
 Assume that {\small$C_{\sigma(t-1)}(A,B)$} is full rank. 
 For a system of the form (\ref{switch}), the minimum input energy required to drive the state from $x(0)$ to $x(t)$ is 
{\small\begin{equation}
 (x(t)-A^tx(0))^{*}W_{\sigma(t-1)}^{-1}(x(t)-A^tx(0)).
\end{equation}}
\end{proposition}
\begin{proof}
Follows from the similar arguments used for the LTI case in {\small\cite{broc}}, Chapter {\small$3$}, Section {\small$22$}, Theorem {\small$1$}.
%  We first prove that 
% \begin{eqnarray}
%  \bar{u}_l= C_{\sigma(t-1)}(A,B)^{*}(W_{\sigma(t-1)})^{-1}(x(t)-A^tx(0))\label{minip}
% \end{eqnarray}
% is the minimum energy input. It is clear from Equation (\ref{ip}) that $\bar{u}_l$ drives the state to $x(t)$. Let $\hat{u}$ be any other input which 
% drives the state to $x(t)$ from $x(0)$.  Note that from Equation (\ref{ip}),
% \begin{eqnarray}
% x(t)-A^tx(0)&=&C_{\sigma(t-1)}(A,B)\hat{u}\label{eq}
% \end{eqnarray}
% We prove that $\hat{u}-\bar{u}_l$ is orthogonal to $\bar{u}_l$. 
% \begin{eqnarray}
% \Rightarrow \bar{u}_l^{*}\hat{u}&=& (x(t)-A^tx(0))^{*}(W_{\sigma(t-1)})^{-1}C_{\sigma(t-1)}(A,B)\hat{u}\nonumber\\
% &=&(x(t)-A^tx(0))^{*}(W_{\sigma(t-1)})^{-1}(x(t)-A^tx(0))\nonumber\\
% &=&\bar{u}_l^{*}\bar{u}_l.
% \end{eqnarray}
% Thus, $\bar{u}_l^{*}(\hat{u}-\bar{u}_l)=0$. 
% \begin{eqnarray}
%  \hat{u}&=&\hat{u}-\bar{u}_l+\bar{u}_l\nonumber\\
%  \Rightarrow \hat{u}^{*}\hat{u}&=& ((\hat{u}-\bar{u}_l)^{*}+\bar{u}_l^{*})((\hat{u}-\bar{u}_l)+\bar{u}_l) \nonumber\\
%  &=&(\hat{u}-\bar{u}_l)^{*}(\hat{u}-\bar{u}_l)+\bar{u}_l^{*}\bar{u}_l.
% \end{eqnarray}
% 
% This shows that $\bar{u}_l$ is the minimum energy input and the minimum energy is $\bar{u}_l^T\bar{u}_l = (x(t)-A^tx(0))^TW_{\sigma(t-1)}^{-1}(x(t)-A^tx(0))$.
\end{proof}

%\section{the controllability Gramian and the Hadamard product}
We now give an expression for the controllability Gramian using the Hadamard product of matrices ({\small\cite{horn}}). 
\begin{assum}\label{assum1}
 We assume that the discrete linear system {\small$x(t+1)=Ax(t)+Bu(t)$} is controllable.
\end{assum}
\begin{assum}\label{assum2}
 We assume that {\small$A$} is diagonalizable. We also assume that no two eigenvalues of {\small$A$} have the same modulus. 
 Furthermore, {\small$0$} is not an eigenvalue of {\small$A$}. %Furthermore, we assume that $A$ is diagonalizable.
\end{assum}
The assumption of diagonalizability was also made in \cite{Pasqualetti} where the decentralized control of discrete LTI systems is considered and also in 
\cite{ol} for discrete LTI systems.

Let $v_1,v_2,\ldots,v_n$ be the left eigenvectors of $A$ and $\lambda_1,\ldots,\lambda_n$ be the corresponding eigenvalues. 
We define the following $n\times (t+1)$ matrix. 
{\small \begin{eqnarray}
  \Lambda_t &:=& 
  \left[ \begin{array}{cccc} 1 & \lambda_1& \ldots &\lambda_1^t\\1 & \lambda_2& \ldots &\lambda_2^t\\.&.&\ldots&.\\.&.&\ldots&.\\ 1 & \lambda_n& \ldots 
  &\lambda_n^t \end{array} \right].\label{lamb}
 \end{eqnarray}}
\begin{definition}
 For a switching signal $\sigma$, we define  
 {\small\begin{eqnarray}
  \Lambda_{\sigma(t)} &:=& 
  \left[ \begin{array}{cccc} \sigma(t) & \sigma(t-1)\lambda_1& \ldots &\sigma(0)\lambda_1^t\\ \sigma(t) & \sigma(t-1)\lambda_2& \ldots &\sigma(0)\lambda_2^t\\.&.&\ldots&.\\.&.&\ldots&.\\
  \sigma(t) & \sigma(t-1)\lambda_n& \ldots 
  &\sigma(0)\lambda_n^t \end{array} \right].\label{lambsig}
 \end{eqnarray}}
 % Note that if $A$ is stable and $t$ tends to infinity, then there are infinitely many columns in $\hat{\Lambda}_{\sigma(t)}$. \\
\end{definition}
 Let {\small$\hat{\Lambda}_{\sigma(t)}$} be the matrix obtained from the above matrix by keeping only non-zero columns. 
 Let {\small$V^{*}=\left[ \begin{array}{cccc} v_1^{*} & v_2^{*} &\ldots & v_n^{*}\end{array}\right]$} be a matrix whose columns are right eigenvectors of $A^{*}$.
\begin{theorem}\label{gramexp}
 Consider a discrete linear system of the form {\small$(\ref{switch})$}. Let {\small$V$} be a non-singular matrix such that rows of {\small$V$} form a set of left eigenvectors of {\small$A$}.  
 %There exists a basis in which 
 Then, choosing rows of {\small$V$} as a basis for {\small$\mathbb{C}^n$}, the controllability Gramian for {\small$(\ref{switch})$} for a switching 
 signal $\sigma$ is given by 
 {\small$(VBB^{*}V^{*}) \circ (\hat{\Lambda}_{\sigma(t)}\hat{\Lambda}_{\sigma(t)}^*)$} 
 (where $\circ$ denotes the Hadamard product and $\hat{\Lambda}_{\sigma(t)}$ is obtained from $(\ref{lambsig})$ by dropping columns with all zeros).
\end{theorem}
\begin{proof}
 Let {\small$V$} be a matrix whose rows are left eigenvectors of {\small$A$}, hence {\small$VA=D_A V$} where {\small$D_A$} is a diagonal matrix having eigenvalues of {\small$A$}. Consider 
 a new basis for {\small$\mathbb{C}^n$} given by the rows of {\small$V$}. Let {\small$\bar{A}=VAV^{-1}, \bar{B}=VB$}. Thus, the controllability Gramian 
 {\small$\bar{W}_{\sigma(t)}=\sum_{i=0}^t
 \sigma(t)\bar{A}^i\bar{B}\bar{B}^*(\bar{A}^*)^i=$ 
 \begin{eqnarray}
 V(\sum_{i=0}^t \sigma(t){A}^iBB^*({A}^*)^i)V^*=\sum_{i=0}^t \sigma(t)V{A}^iBB^*({A}^*)^iV^*=\nonumber\\
 \sum_{i=0}^t \sigma(t){D_A}^iVBB^*V^*({D_A}^*)^i  
 =(VBB^{*}V^{*}) \circ (\hat{\Lambda}_{\sigma(t)}\hat{\Lambda}_{\sigma(t)}^*)\nonumber\end{eqnarray}}. 
\end{proof}

\section{Probabilistic state transfer and state estimation}
\subsection{Probabilistic state transfer}
 We use properties of the Hadamard product to obtain the necessary and sufficient conditions for an arbitrary state transfer of $(\ref{switch})$. 
We need the following result from 
\cite{horn}.
\begin{lemma}\label{hdposlemma}
 If {\small$X,Y$} are positive semi-definite, then so is {\small$X\circ Y$}. If, in addition, {\small$Y$} is positive definite and 
 {\small$X$} has no diagonal entry equal to {\small$0$},
then {\small$X\circ Y$} is positive definite. In particular, if both {\small$X$} and {\small$Y$} are positive
definite, then so is {\small$X\circ Y$}.
\end{lemma}
\begin{proof}
 We refer the reader to Theorem $5.2.1$ of \cite{horn}. 
\end{proof}
\begin{theorem}\label{singleipcdns}
 Let $p$ be the probability that $\sigma(t)=1$ where $t \in \mathbb{N}$. 
 Consider a single input system of the form (\ref{switch}). Suppose $A$ is non-singular and $(A,B)$ controllable. Then, 
 \begin{enumerate}
  \item ${W}_{\sigma(t)}$ is positive definite if and only if $\sigma$ is non-zero for at least $n$ time instances.
  \item The probability that an arbitrary state transfer from $x_0\in \mathbb{C}^n$ to $x_f\in \mathbb{C}^n$ is possible for $(\ref{switch})$ is given 
  by {\small$P(T)=\sum_{i=n}^T{T \choose i} p^i(1-p)^{T-i}$}.  
 \end{enumerate}
\end{theorem}
\begin{proof}
 Note that after a change of basis, %the controllability Gramian is 
 {\small$W_{\sigma(t)}=(VBB^*V^*) \circ (\hat{\Lambda}_{\sigma(t)}\hat{\Lambda}_{\sigma(t)}^*)$}. 
 By Assumption \ref{assum1}, {\small$VBB^*V^*$} has non-zero diagonal entries. 
 With reference to the Lemma \ref{hdposlemma}, {\small$\bar{W}_{\sigma(t)}$} is positive definite if the matrix {\small$\hat{\Lambda}_{\sigma(t)}\hat{\Lambda}_{\sigma(t)}^*$} is positive definite. 
 For a single input system, {\small$VBB^*V^*$} has rank one. 
 It is shown in \cite{horn} that rank{\small$(X \circ Y) \le$} rank$(X)$rank$(Y)$. Thus, for {\small$(VBB^*V^*) \circ (\hat{\Lambda}_{\sigma(t)}\hat{\Lambda}_{\sigma(t)}^*)$} to be of full rank, 
 {\small$(\hat{\Lambda}_{\sigma(t)}\hat{\Lambda}_{\sigma(t)}^*)$} must be full rank. By Assumption \ref{assum1}, {\small$(\ref{dlin})$} is controllable and 
 by Assumption \ref{assum2}, {\small$A$} is diagonalizable. Hence, {\small$A$} must 
 have distinct eigenvalues. Thus, $(\hat{\Lambda}_{\sigma(t)}\hat{\Lambda}_{\sigma(t)}^*)$ is full rank 
 {\small$\Leftrightarrow$} {\small$\sigma$} is non-zero for at least {\small$n$} time instances. The probability of having {\small$n$} ones and 
 {\small$T-n$} zeros in {\small$T$} time instances is 
 {\small${T \choose n} p^n(1-p)^{T-n}$}. Thus, the second statement follows. 
\end{proof}

\begin{remark}
 For multi-input systems or the case where {\small$A$} is diagonalizable with repeated real eigenvalues, it could happen that 
 {\small$W_{\sigma(t)}$} is full rank but  
 both {\small$VBB^*V^*$} and {\small$\hat{\Lambda}_{\sigma(t)}$} are not full rank. Furthermore, when {\small$A$} has repeated eigenvalues, 
 {\small$\hat{\Lambda}_{\sigma(t)}$} is never 
 full rank. Thus, we can not apply Theorem \ref{singleipcdns} to characterize signals for which {\small${W}_{\sigma(t)}$} is positive definite. \QEDopen
\end{remark}
We give the following result from \cite{won} which is required in our next result for the case of repeated eigenvalues of {\small$A$}. 
\begin{proposition}\label{wh}
 Suppose {\small$(A, B)$} is controllable. Let {\small$m$} be the number of inputs and {\small$k$} be the cyclic index of {\small$A$} (i.e. the number of invariant factors of {\small$A$}). %$(m\ge k).$ 
 Let {\small$\mathcal{B}$} be the column span of {\small$B=\left[ \begin{array}{cccc} b_1 & b_2& \ldots &b_m \end{array} \right]$}. 
 Let {\small$\mathcal{B}_i=\langle
 b_1,\ldots,b_i\rangle$ 
 $(1\le i \le m)$} where 
 {\small$\langle b_1,\ldots,b_i\rangle$} denotes the span of the corresponding vectors. 
 Let {\small$\alpha_1, \ldots , \alpha_k$} be the invariant factors of {\small$A$}. 
 Then, there exists {\small$A-$}invariant subspaces {\small$\mathcal{X}_i \subset \mathbb{C}^n$} and subspaces {\small$\mathcal{B}_i \subset \mathcal{B}$} 
 such that 
 \begin{itemize}
  \item {\small$\mathbb{C}^n = \mathcal{X}_1 \oplus \ldots \oplus \mathcal{X}_k$}.
  \item {\small$A$} restricted to {\small$\mathcal{X}_i$} is cyclic with minimal polynomial {\small$\alpha_i$}.
  \item {\small$\langle\mathcal{B}_i,A\mathcal{B}_i,\ldots,A^{n-1}\mathcal{B}_i\rangle = \mathcal{X}_1 \oplus \ldots \oplus \mathcal{X}_i$}.
 \end{itemize}
\end{proposition}
\begin{proof}
 We refer the reader to Theorem {\small$1.2$} of {\small\cite{won}}. 
\end{proof}
From Proposition {\small$\ref{wh}$} , it can be shown that there exists a basis such that {\small$A, B$} can be transformed as %into the following form
{\tiny\begin{eqnarray}
 A=   \left[ \begin{array}{cccc} A_1 & 0& \ldots &0\\0 & A_2& \ldots &0\\.&.&\ldots&.\\.&.&\ldots&.\\
  0 & 0& \ldots  &A_k \end{array} \right], B = \left[ \begin{array}{ccccc} b_{11} & b_{12}& \ldots &b_{1k}&*\\0 & b_{22}& \ldots &b_{2k}&*\\.&.&\ldots&.&.\\.&.&\ldots&.&.\\
  0 & 0& \ldots  &b_{kk}&* \end{array} \right]\label{can}
\end{eqnarray}}
where {\small$(A_i, b_{ii})$} is controllable ({\small\cite{won}}, page {\small$44$}) (note that entries {\small$*$} in the matrix {\small$B$} above denote all the remaining columns of {\small$B$}).
\begin{theorem}\label{repeig}
 Let $\alpha_i$ $(1\le i \le k)$ be the invariant factors of $A$ such that $\alpha_{i+1}|\alpha_i$ (where $k$ is the cyclic index of $A$). 
 Let $n_1$ be the degree of the minimal polynomial $\alpha_1$. Let $m$ be the number of inputs such that $k\le m \le n$.
 If the switching signal $\sigma$ has at least $n_1$ non-zero entries, then ${W}_{\sigma(t)}$ is positive definite.
\end{theorem}
\begin{proof}
 For simplicity, we assume that there are just two invariant factors $\alpha_1$ and $\alpha_2$. The general case follows in exactly similar manner. 
 Let  
 {\small\begin{eqnarray}
 A= \left[ \begin{array}{cc} A_1 & 0\\0 & A_2 \end{array} \right],
 B = \left[ \begin{array}{cccc} b_{11} & b_{12}&*\\0 & b_{22}&* \end{array} \right]\label{new1}\\
 \hat{\Lambda}_{\sigma(t)}=\left[ \begin{array}{c} \hat{\Lambda}_{\sigma(t)}(A_1) \\ \hat{\Lambda}_{\sigma(t)}(A_2) \end{array} \right],V= \left[ \begin{array}{cc} V_1 &0\\0& V_2 \end{array} \right]. \label{new2}
\end{eqnarray}}
Let {\small$P=\hat{\Lambda}_{\sigma(t)}\hat{\Lambda}_{\sigma(t)}^*$}. 
Note that %$\hat{\Lambda}_{\sigma(t)}^T\hat{\Lambda}_{\sigma(t)}=P=$
{\small\begin{eqnarray}
 P=\left[ \begin{array}{cc} \hat{\Lambda}_{\sigma(t)}(A_1)\hat{\Lambda}_{\sigma(t)}^*(A_1) & \hat{\Lambda}_{\sigma(t)}(A_1)\hat{\Lambda}_{\sigma(t)}^*(A_2)\\
 \hat{\Lambda}_{\sigma(t)}(A_2)\hat{\Lambda}_{\sigma(t)}^*(A_1)&\hat{\Lambda}_{\sigma(t)}(A_2)\hat{\Lambda}_{\sigma(t)}^*(A_2)\end{array} \right].\label{new3}
\end{eqnarray}}
%Let 
From Theorem $\ref{gramexp}$, %and Remark \ref{gramexpac}, %we have 
{\small\begin{eqnarray}
W_{\sigma(t)}&=&(VBB^*V^*) \circ (\hat{\Lambda}_{\sigma(t)}\hat{\Lambda}_{\sigma(t)}^*). \label{new4}
\end{eqnarray} }
Let $b_1=\left[ \begin{array}{cc} b_{11}\\0 \end{array} \right], b_2=\left[ \begin{array}{c} b_{12}\\b_{22} \end{array} \right]$ and 
$B_3=\left[ \begin{array}{c} *\\* * \end{array} \right]$. Observe that 
{\small\begin{eqnarray}
 BB^*=b_1b_1^*+b_2b_2^*+B_3B_3^*. \nonumber
\end{eqnarray}}
(Note that if $m=k$ i.e., if $m=2$ in this case, then we define $B_3=0$.)
 Using this decomposition of $BB^*$, we get {\small
\begin{eqnarray}
 W_{\sigma(t)}&=&\sum_{i=0}^{t}\sigma(t-i)A^iBB^*(A^*)^i\nonumber\\
 &=&\sum_{i=0}^{t}\sigma(t-i)A^ib_1b_1^*(A^*)^i+\sum_{i=0}^{t}\sigma(t-i)A^ib_2b_2^*(A^*)^i\nonumber\\
 &&+\sum_{i=0}^{t}\sigma(t-i)A^iB_3B_3^*(A^*)^i\label{new5}.
\end{eqnarray}}

From $(\ref{new1})$, $(\ref{new2})$, $(\ref{new3})$, $(\ref{new4})$ and $(\ref{new5})$,
{\small\begin{eqnarray}
 W_{\sigma(t)}=M_1+M_2+M_3
\end{eqnarray} }
 where
 {\small\begin{eqnarray}
 M_1&=&\sum_{i=0}^{t}\sigma(t-i)A^ib_1b_1^*(A^*)^i\nonumber\\
 &=&\left[ \begin{array}{cc} (V_1b_{11}b_{11}^*V_1^*)\circ (P_{11})& 0\\0 & 0 \end{array} \right]\label{m1},\\
 M_2&=&\sum_{i=0}^{t}\sigma(t-i)A^ib_2b_2^*(A^*)^i=\nonumber
 \end{eqnarray} }
 {\small\begin{eqnarray}
 \left[ \begin{array}{cc} (V_1b_{12}b_{12}^*V_1^*)\circ (P_{11})& (V_1b_{12}b_{22}^*V_2^*)\circ (P_{12})\\
 (V_2b_{22}b_{12}^*V_1^*)\circ (P_{12}^*) & (V_2b_{22}b_{22}^*V_2^*)\circ (P_{22}) 
 \end{array} \right]\label{m2}
 \end{eqnarray}}
 where 
 {\small\begin{eqnarray}
  P_{11}&=& \hat{\Lambda}_{\sigma(t)}(A_1)\hat{\Lambda}_{\sigma(t)}^*(A_1)\nonumber\\
  P_{12}&=&\hat{\Lambda}_{\sigma(t)}(A_1)\hat{\Lambda}_{\sigma(t)}^*(A_2)\nonumber\\
  P_{12}^*&=&\hat{\Lambda}_{\sigma(t)}(A_2)\hat{\Lambda}_{\sigma(t)}^*(A_1)\nonumber\\
  P_{22}&=&\hat{\Lambda}_{\sigma(t)}(A_2)\hat{\Lambda}_{\sigma(t)}^*(A_2)\nonumber
 \end{eqnarray}}
 and {\small$M_3=\sum_{i=0}^{t}\sigma(t-i)A^iB_3B_3^*(A^*)^i$}.
 
Note that if $\sigma$ has at least $n_1$ non-zero entries,  
%the condition mentioned in the statement of the theorem is satisfied, 
then both {\small$\hat{\Lambda}_{\sigma(t)}(A_1)$} and {\small$\hat{\Lambda}_{\sigma(t)}(A_2)$} are 
full row rank. 
Since {\small$M_1$, $M_2$} and {\small$M_3$} are positive semidefinite, {\small$vW_{\sigma(t)}v^*=0$ $\Leftrightarrow$ $vM_1v^*=0$, $vM_2v^*=0$} and 
{\small$vM_3v^*=0$}. 
One can write {\small$v=v_1\oplus v_2$} such that {\small$v_1 \in $} row span of {\small$V_1$} and {\small$v_2 \in$} row span of {\small$V_2$}. 
Note that {\small$vM_1v^*=0$ $\Leftrightarrow$ $v_1=0$}. 
Suppose {\small$v_1=0$} and {\small$v=0\oplus v_2$}, then {\small$vM_2v^* = 0$ $\Leftrightarrow$ $v_2= 0$}. 
Thus, {\small$vW_{\sigma(t)}v^*=0$} is possible only for 
{\small$v=0$}. Therefore, if {\small$\sigma$} has at least {\small$n_1$} non-zero entries, then {\small$W_{\sigma(t)}$} is positive definite. The general case 
for {\small$k$} invariant factors follows using similar arguments. Let {\small$M_j = \sum_{i=0}^{t}\sigma(t-i)A^ib_jb_j^*(A^*)^i$} for {\small$1\le j \le k$} and 
{\small$M_{k+1}=\sum_{i=0}^{t}\sigma(t-i)A^iB_{k+1}B_{k+1}^*(A^*)^i$}. 
We write {\small$BB^* = b_1b_1^*+\ldots+b_kb_k^*+B_{k+1}B_{k+1}^*$} and {\small$W_{\sigma(t)}= M_1+\ldots+M_k+M_{k+1}$} 
as a sum of $k+1$ positive semidefinite matrices and apply the same trick used above. 
If {\small$m=k$}, then we define {\small$M_{k+1}=0$} and the same arguments work. 
\end{proof}
\begin{example}\label{ex2}
Let 
 {\small\begin{eqnarray}
   A= \left[ \begin{array}{ccccc} 2& 0& 0& 0& 0\\0& 4& 0& 0& 0\\0& 0& 5& 0& 0\\0& 0& 0& 2& 0\\0& 0& 0& 0& 4 \end{array} \right], B= 
  \left[ \begin{array}{cc} 1 &1\\1& 0\\1& 1\\0& 1\\0& 1 \end{array} \right].\nonumber
 \end{eqnarray}}Observe that 
 {\small$A$} has two invariant factors of degree {\small$3$} and {\small$2$} respectively. Thus, {\small$n_1=3$, $n_2=2$}. 
 %Suppose no more that two consecutive packet losses are allowed. 
 Let {\small$\sigma = \{1,0,0,1,0,0,1\}$}. We observe that the condition of Theorem \ref{repeig} is 
 satisfied and {\small$W_{\sigma}$} is positive definite. \QEDopen
\end{example}
Note that using the notation used in the above theorem, we can write {\small$BB^T=\sum_{i=1}^kb_ib_i^* +B_{k+1}B_{k+1}^*$}. The controllability Gramian is 
given by 
{\small\begin{eqnarray}
 W_{\sigma(t)}&=&\sum_{i=1}^{k+1}M_i\label{gramsum}
\end{eqnarray}}where {\small$M_i$} is as defined in the proof of Theorem \ref{repeig}. The following corollary gives necessary and sufficient conditions for controllability 
of multi-input systems.
\begin{corollary}\label{necsuf}
 For a multi-input system {\small$(A,B)$, ${W}_{\sigma(t)}$} is positive definite for a signal {\small$\sigma$} at time {\small$t$ $\Leftrightarrow$ $\cap_{i=1}^{k+1}$} 
 ker{\small$(M_i) =0$}.
\end{corollary}
\begin{proof}
 Note that since {\small$M_i\ge0$ ($i=1,\ldots,k+1$)}, 
 from Equation {\small$(\ref{gramsum})$}, ker{\small$(W_{\sigma(t)}) =0 $} implies that {\small$\cap_{i=1}^{k+1}$} ker{\small$(M_i) =0$} and conversely. 
\end{proof}
Thus, given a signal $\sigma$ and a fixed time $t$, we obtain necessary and sufficient conditions for {\small$W_{\sigma(t)}$} to be positive definite for 
single input as well as multi-input systems. 
Note that %if the number of inputs $m$ are comparable with the dimension $n$ of the state space, then 
the sufficient condition of Theorem \ref{repeig} is not necessary.  
For example, in Example \ref{ex2}, if {\small$B=I_5$} then for any non-trivial {\small$\sigma$, $W_{\sigma}$} is positive definite for 
{\small$\sigma$} even when the number of its non-zero entries are strictly less than the degree of the minimal polynomial of {\small$A$}.  
\begin{lemma}\label{necc}
 Let $m$ be the number of inputs of $(\ref{switch})$ and $n$ be the dimension of the state space. Let $l = \lceil \frac{n}{m} \rceil$. Then,  
 $W_{\sigma}$ is singular if the number of non-zero entries of $\sigma$ are strictly less than $l$.
\end{lemma}
\begin{proof}
 If the number of non-zero entries of $\sigma$ are strictly less than $l$, then the number of columns of {\small$C_{\sigma}(A,B)$} are less than $n$. Hence,  
 {\small$W_{\sigma}$} is singular. 
\end{proof}

The above lemma says that for {\small$W_{\sigma}$} to be non-singular, $\sigma$ must have at least $l$ non-zero entries. Thus, we have a necessary condition 
of $\sigma$ which can be checked by looking at the entries of $\sigma$. Again this necessary condition is not sufficient as the following example shows. 
\begin{example}
Let {\small$A= \left[ \begin{array}{cccc} 2& 0& 0& 0\\0& 4& 0& 0\\0& 0& 5& 0\\0& 0& 0& 2\end{array} \right], B= 
  \left[ \begin{array}{cc} 1 &0\\1& 0\\1& 0\\0& 1 \end{array} \right]$}.
 Since {\small$n=4$} and {\small$m=2$, $\lceil \frac{n}{m} \rceil = 2$}. Observe that if {\small$\sigma$} has two non-zero entries 
 (say {\small$\sigma=\{1,0,0,1\}$}), 
 {\small$W_{\sigma}$} still remains singular. \QEDopen
\end{example}

\begin{remark}\label{multiipcase}
 Let $n_{\sigma}$ be the number of non-zero entries of $\sigma$ and $n_1$ be the degree of the minimal polynomial of the system matrix {\small$A$}. Then we observe that 
\begin{itemize}
 \item if {\small$n_{\sigma}<\lceil \frac{n}{m} \rceil$}, then {\small$W_{\sigma}$} is singular.
 \item if {\small$n_{\sigma}\ge n_1$}, then {\small$W_{\sigma}$} is positive definite.
 \item if {\small$\lceil \frac{n}{m} \rceil \le n_{\sigma} < n_1$}, then we need Corollary \ref{necsuf}. \QEDopen
\end{itemize}
\end{remark}

\begin{theorem}\label{multi-ip-prob}
 Let $p$ be the probability 
 that $\sigma(t)=1$ for $t\in\mathbb{N}$. 
 Consider a multi-input system of the form $(\ref{switch})$. Assume that {\small$A$} is non-singular and {\small$(A,B)$} controllable. Let $m$ be the number of inputs and 
 $n_1$ be the degree of the minimal polynomial of {\small$A$}. Let $n$ be the dimension of state space and $l=\lceil \frac{n}{m} \rceil$. Let {\small$P(T)$} be the 
 probability that an arbitrary state transfer is possible at time {\small$T$}. 
 %Let $T$ be the given time instant and let $P_c(T)$ be the probability that $(\ref{switch})$ is controllable at time $T$. 
 Then, 
 {\small\begin{equation}
 \sum_{i=n_1}^T{T\choose i}p^{i}(1-p)^{T-i}\le P(T)\le 1-\sum_{i=1}^{l-1}{T\choose i}p^{i}(1-p)^{T-i}. \label{bds} 
 \end{equation}}
 \end{theorem}
\begin{proof}
 It follows from Theorem \ref{repeig} that if the switching signal $\sigma$ has at least $n_1$ non zero entries, then $(\ref{switch})$ is controllable. 
 Thus, {\small$\sum_{i=n_1}^T{T\choose i}p^{i}(1-p)^{T-i}\le P(T)$}. From Lemma $\ref{necc}$, it is clear that for a switching signal $\sigma$, an arbitrary state transfer is not possible for 
 $(\ref{switch})$ if the number of non zero entries of the switching signal $\sigma$ is strictly less than $l$. Therefore, 
 {\small$P(T)\le 1-\sum_{i=1}^{l-1}{T\choose i}p^{i}(1-p)^{T-i}$}. 
\end{proof}
\begin{definition}\label{contrsigs}
Let {\small$\mathcal{S}^{T}_{c}$} be the set of switching signals for which an arbitrary state transfer 
is possible in time $T$. Let 
{\small$\mathcal{S}^T_{\ge n}$} denote the set of switching signals of length {\small$T$} with the number of non zero entries greater than or equal to $n$ and 
{\small$\mathcal{S}^T_{\le n}$} denote the set of switching signals of length {\small$T$} with the number of non zero entries less than or equal to $n$.  
\end{definition}
\begin{remark}\label{enumsigs}
 It is clear that for single input systems, {\small$\mathcal{S}^T_{c}=\mathcal{S}^T_{\ge n}$}. For multi-input systems, we do not have an exact 
 enumeration of {\small$\mathcal{S}^T_{c}$}. However, {\small$\mathcal{S}^T_{\ge n_1}\subset \mathcal{S}^T_{c}\subset \mathcal{S}^T\setminus 
 \mathcal{S}^T_{\le \lceil \frac{n}{m} \rceil}$} where $m$ is the number of inputs. \QEDopen
\end{remark}
\begin{definition}\label{avgeng}
Let {\small$P(\sigma)$} be the probability of occurrence of $\sigma$. %Consider an event $E$ that an arbitrary state transfer 
%is possible at time $T$. 
The average control input energy to go from $x_0$ to $x_f$ in {\small$T$} time steps over the set {\small$\mathcal{S}^{T}_{c}$} is %when $E$ has occurred is
 {\small\begin{eqnarray}
 E^i_{av}(x_0,x_f,T):=%\frac{1}{Z}
 \sum_{\sigma \in \mathcal{S}^T_{c}}P(\sigma)(x_f-x_0)^*W_{\sigma(T)}^{-1}(x_f-x_0). \label{aveng}
\end{eqnarray}} 
\end{definition}
\begin{theorem}\label{actplace}
 Let $n$ be the degree of the characteristic polynomial of {\small$A$} and let $n_1$ be the degree of minimal polynomial of {\small$A$}. 
 Let {\small$P(\sigma)$} be the probability of occurrence of $\sigma$. Then, 
 \begin{itemize}
  \item  For single input systems {\small$(\ref{switch})$, $E^i_{av}(x_0,x_f,T)=$}
  {\small\begin{eqnarray}
   \sum_{\sigma \in \mathcal{S}^T_{\ge n}}P(\sigma)(x_f-x_0)^*W_{\sigma(T)}^{-1}(x_f-x_0).
  \end{eqnarray}}
  \item  For multi-input systems {\small$(\ref{switch})$, $E^i_{av}(x_0,x_f,T)\ge$}
  {\small\begin{eqnarray}
   \sum_{\sigma \in \mathcal{S}^T_{\ge n_1}}P(\sigma)(x_f-x_0)^*W_{\sigma(T)}^{-1}(x_f-x_0).
  \end{eqnarray}}
 \end{itemize}
\end{theorem}
\begin{proof}
 Follows from Definition \ref{avgeng}, Theorem \ref{singleipcdns} and Theorem \ref{multi-ip-prob}. 
\end{proof}
Next, we consider the dual state estimation problem to obtain the probability that state estimation 
is possible from measured outputs with time going from {\small$0$} to {\small$T$}. 
\subsection{Probabilistic state estimation}
The next natural step is to 
 consider non-idealities in the transmission of measurements obtained by sensors over a communication network for discrete LTI systems 
\cite{JungersKunduHeemels}, \cite{sin,mo}. 
Consider the following discrete linear system subject to packet dropouts 
{\small\begin{eqnarray}\label{switch1}
 x(t+1)&=&Ax(t)\nonumber\\
 y(t) &=&\begin{cases} Cx(t), \mbox{ if } \sigma(t)=1,\\
 \emptyset, \;\;\;\;\;\;\;\mbox{ if } \sigma(t)=0. \end{cases}
\end{eqnarray}}where the switching signal $\sigma$ is random signal with Bernoulli distribution, taking values $0$ and $1$. 
The observability Gramian associated with discrete LTI systems is defined as {\small$W_t^{o}:=\sum_{i=0}^t(A^{*})^iC^{*}CA^i$} (\cite{hespanha1}). 
If {\small$W^o_{t}$} is singular, then the states in the null space of {\small$W^o_{t}$} are unobservable. 
%  It turns out that a nearly singular observability Gramian results in poor results for the state estimation algorithms \cite{saltik}. 
% Suppose $y=\left[ \begin{array}{ccccc}  y(0)^{*}&y(1)^{*}&\cdots & & y(t)^{*}\end{array}\right]^{*}$. It follows that 
%  $y^{*}y=x^{*}(0)W^o_{t}x(0)$. 
%  A state lying in the direction associated with the eigenvector corresponding to the minimum eigenvalue
%  of the observability Gramian corresponds to the least output energy among all states hence it is the least observable state  \cite{ben}. 
%  %The minimum eigenvalue of the OG gives the energy of the output corresponding to the least observable state and its 
%  The inverse of the minimum eigenvalue of the observability Gramian gives the maximum estimation 
%  uncertainty whereas the trace of the inverse of the observability Gramian gives the average estimation uncertainty \cite{hin}.  

 Define the associated observability Gramian for $(\ref{switch1})$ as
 {\small\begin{eqnarray}
  W^{o}_{\sigma(t)}&:=&\sum_{i=0}^{t}\sigma(i)(A^{*})^iC^{*}CA^i.\nonumber
 \end{eqnarray}}
 Note that, the observability matrix {\small$O_{\sigma(t)}(C,A)=$} 
{\small\begin{eqnarray}
  \left[ \begin{array}{cccc} \sigma(0)C^{*} & \sigma(1)(CA)^{*} &\cdots  & \sigma(t)(CA^{t})^{*}\end{array}\right]^{*}.\label{obsmat}        
\end{eqnarray}}
Let {\small$y=\left[ \begin{array}{ccccc}  y(0)^{*}&y(1)^{*}&\cdots & & y(t)^{*}\end{array}\right]^{*}$} be a vector of observed outputs. 
Then using {\small$y(t)=C\sigma(t)x(t)$}, we get {\small$x(0)=(W^{o}_{\sigma(t)})^{-1}O_{\sigma(t)}^{*}(C,A)y$}. 
 The exact counterpart of the arbitrary state transfer Theorems (Theorem \ref{gramexp}, Theorem \ref{singleipcdns}, Theorem \ref{repeig}, 
  Theorem \ref{multi-ip-prob}), hold for the observability Gramian and the state estimation for a given switching signal. 
  %(we refer the reader to Remark $\ref{opt-act}$). 
  %Note that {\small$y^*y=x^*(0)W^{o}_{\sigma(t)}x(0)$}. 
  
  Let {\small$\mathcal{S}^T_o$} be the set of switching signals for which the observability Gramian 
  becomes full column rank. Using counterparts of Theorem $\ref{singleipcdns}$ and Theorem $\ref{multi-ip-prob}$, we can find the probability that 
  state estimation is possible from the measured outputs $y(0),\ldots,y(T)$ with packet dropouts. 
  
  Note that {\small$y^*y=x^*(0)W^{o}_{\sigma(t)}x(0)$}. The average output energy for a fixed time {\small$T$} over the set {\small$\mathcal{S}^T_o$} is defined as 
  {\small \begin{eqnarray}
            E^o_{av}(x(0),T)=\sum_{\sigma \in \mathcal{S}^T_{o}}P(\sigma)x^*(0)W^o_{\sigma(T)}x(0).
          \end{eqnarray}
}
It follows that the counterpart of Theorem $\ref{actplace}$ holds. %We define 

\section{Probabilistic measures for optimal actuator/sensor placement}
We use the results obtained in the previous section to address optimal actuator/sensor placement problems for linear systems with packet dropouts. 
%We now define a probabilistic measure for optimal actuator placement problem in the following definition. 
The following definition gives a new probabilistic measure for our models. 
\begin{definition}\label{probmet}
 Let {\small$T$} be fixed and 
%  {\small\begin{eqnarray}
%    W_{av}^{-1}(T):= \sum_{\sigma \in \mathcal{S}^T_{c}}P(\sigma)^{-1}W_{\sigma(T)}^{-1}.\label{avggraminv}
%  \end{eqnarray}}be the average inverse of the controllability Gramian over all switching signals in the set $\mathcal{S}^T_{c}$. 
 {\small\begin{eqnarray}
   \mu(T):= \sum_{\sigma \in \mathcal{S}^T_{c}}P(\sigma)\mbox{det}(W_{\sigma(T)})\label{avggramdetg}%\\
   %\mu_2(T):= \mbox{det}(\sum_{\sigma \in \mathcal{S}^T_{c}}P(\sigma)W_{\sigma(T)})\label{avggramdeta}
 \end{eqnarray}}
\end{definition}
% Note that just like the Gramian the probability of a packet dropout is also inversely related to the energy i.e., lower the 
% probability of a dropout, higher is the required amount of energy. This explains the inverse factor $P(\sigma)^{-1}$ in the above definition. 
\begin{remark}\label{opt-act}
%We can use the trace{\small$(W_{av}^{-1}(T))$} as a metric for optimal actuator placement problem. 
% The actuators which minimize 
% the maximum eigenvalue or the trace of $W_{av}^{-1}(T)$ are optimal actuators. 
%
We can use {\small$(\ref{avggramdetg})$} as a controllability metric for the optimal actuator placement problem. 
Equation {\small$(\ref{avggramdetg})$} gives average volume 
reached over all signals in {\small$\mathcal{S}^T_{c}$} using unit energy inputs.  
Note that by Remark $\ref{enumsigs}$, for single input systems, {\small$\mathcal{S}^T_{c}=\mathcal{S}^T_{\ge n}$}. Hence, for single input systems, 
{\small$\mu(T)= \sum_{\sigma \in \mathcal{S}^T_{\ge n}}P(\sigma)\mbox{det}(W_{\sigma(T)})$} thus, one can obtain the optimal solution as {\small$\mathcal{S}^T_{\ge n}$} 
can be enumerated for a fixed time instance {\small$T$}. 

For multi-input systems, {\small$\mathcal{S}^T_{\ge n_1}\subset \mathcal{S}^T_{c}$}. Hence, 
{\small\begin{eqnarray}
\hat{\mu}(T):=\sum_{\sigma \in \mathcal{S}^T_{n_1}}P(\sigma)\mbox{det}(W_{\sigma(T)})\le \mu(T).  
\end{eqnarray}}Thus, we can use {\small$\hat{\mu}(T)$} as a lower bound 
for optimal actuator placement problem for multi-input systems. In other words, we consider the sub-optimal solution for multi-input systems obtained 
by considering {\small$\hat{\mu}(T)$} as a selection criterion. 
% Note that it is shown in \cite{Summers} (Theorem $5$) that the logarithm of the the determinant of the  controllability Gramian considered as a function of the set of actuators is 
% a sub-modular function. It is clear from our definition that the trace of $\mu(T)$ is a sub-modular function of the set of actuator inputs. 
% Therefore, it follows from \cite{Summers} that although the problem is NP-hard, we can use greedy heuristics used in \cite{Summers} to solve the optimal actuator placement problem. 
\QEDopen
% 
% In the same spirit, we may use $(\ref{avgtr})$ as a selection measure for the optimal actuator placement problem. \QEDopen
%In future, we want to enumerate $\mathcal{S}^T_{c}$ for multi-input systems
%  The average energy required to go from $x_0$ to $x_f$ can be used as a selection criterion for optimal actuator 
%  placement problem where the initial and the terminal states are fixed. For single input cases, we can compute the average energy exactly. Thus, actuator with the lowest average energy is naturally the 
%  optimal actuator. For multi-input systems, we give a lower bound on the average energy. In this case, one can use a heuristic approach of selecting the 
%  set of actuators for which the lower bound is minimum among the set of all actuators. 
%  
%  Consider the problem of control of networks where the system matrix $A$ is given by the graph Laplacian (\cite{fitch}). 
%  Consider the same packet dropout model for these networks. One can consider a similar optimal actuator placement 
%  problem for control of networks and use the average energy as a selection criterion. 
\end{remark}
The following example demonstrates that the choice of an optimal actuator could be different for classical systems and systems with packet dropouts. 
It can be observed that the probabilities of packet dropouts associated with the communication network for each actuator plays a role in 
deciding the optimal actuator. 
\begin{example}\label{ex3}
%  Actuator placement for a given set of actuators for a classical discrete LTI system. Now attach different probabilities for dropouts for each 
%  set of actuators and find the optimal set. compare the two choices. Try this for both single and multi input cases. mention computational 
%  challenges somewhere as future work.
% 
Let {\small
\begin{equation}
 A= .04*\left[ \begin{array}{ccc} 2 & 16& 4\\5& 10& 15\\1 &20& 12 \end{array} \right], B= [b_1\; b_2\; b_3]=
  \left[ \begin{array}{ccc} 1&5&2\\4&1&1\\2&4&3 \end{array} \right].\nonumber
\end{equation}} Note that {\small$(A,b_i)$} is controllable for {\small$i=1,2,3$}. 
  Suppose we want to choose an optimal 
  actuator from {\small$b_1,b_2,b_3$}. Let {\small$T=3$}, and let det{\small$(W_T(A,b))$}  be the controllability metric. 
  The values of the determinant of the Gramian for three actuators are {\small$9.1718,315.2886$} and {\small$2.6837$} respectively, implying that 
   {\small$b_2$} gives the optimal actuator for the classical case. 
  
  Now suppose each actuator is connected to the system by a different communication network. Let {\small$p_1=0.8$, $p_2=0.4$} and {\small$p_3=0.5$} be the probabilities 
  of packet dropouts corresponding to the three actuators {\small$b_1,b_2$} and {\small$b_3$} respectively. Using Theorem {\small$\ref{singleipcdns}$}, 
%   {\small$W_{\sigma(t)}>0$} 
%   $\Leftrightarrow$ number of non zero entries of {\small$\sigma$} are greater than or equal to three. Now using 
  with {\small$(\ref{avggramdetg})$} as a 
  controllability metric, the values of {\small$\mu(3)$} for the three actuators are {\small$4.6959,2.5223$} and {\small$0.3355$} respectively. Thus, 
  {\small$b_1$} is the optimal actuator. 

  This demonstrates the role of probabilities of packet dropouts in the communication network deciding the choice of the optimal actuator. 
  %   Suppose we want to choose two optimal actuators. %We get an exact optimal solution in this particular example. 
%   Let $B_1=[b_1\; b_2],B_2=[b_2\; b_3],B_3=[b_1\; b_3]$. For the classical case $B_1$ gives the optimal actuator. Now for the dropout case, 
%   note that $n_1=3$ in this case. Since $m=2$, $\lceil\frac{n}{m}\rceil=2$. Therefore, by Remark $\ref{multiipcase}$, $W_{\sigma}(t)$ is singular 
%   if the number of non zero entries in $\sigma$ is less than $2$. We observe that even for this case, $B_1$ gives the optimal solution. 
\QEDopen
\end{example}
\begin{remark}\label{sensplace1}
Note that for single output systems, 
{\small$\mathcal{S}^T_o=\mathcal{S}^T_{n}$} and for multi-output systems, {\small$\mathcal{S}^T_{n_1}\subset \mathcal{S}^T_o$}. 
We can define similar probabilistic measure as mentioned in Definition {\small$\ref{probmet}$} and Remark {\small$\ref{opt-act}$} for the optimal sensor placement problem. 
It is clear from Example {\small$\ref{ex3}$} that with the probabilistic measure, the choice of the optimal sensor is different from the classical 
observability measure. 
\QEDopen
% the average inverse of the observability Gramian for a given time $T$ as 
% {\small\begin{eqnarray}
%    (W_{av}^o)^{-1}(T):= \sum_{\sigma \in \mathcal{S}^T_{o}}P(\sigma)(W^o_{\sigma(T)})^{-1}.\label{avggraminv1}
%  \end{eqnarray}}
% As mentioned in Remark $\ref{opt-act}$, we do optimal sensor placement in an analogous manner.
 %For multi-output systems
\end{remark}
\begin{remark}
 Another possible controllability metric is 
 {\small$\gamma(T):=\sum_{\sigma \in \mathcal{S}^T_{c}}P(\sigma)\mbox{trace}(W_{\sigma(T)})$}. We can similarly compare 
 actuator placement problems for both classical and probabilistic models. In future, we wish to generalize some of the 
 classical controllability metrics for models considered here. 
\end{remark}

\subsection{Feedback laws and LQR}
Consider the following model
{\small\begin{eqnarray}\label{switchcomb}
 x(t+1)&=&Ax(t)+B\sigma(t)u(t)\nonumber\\
 y(t) &=& C\sigma(t)x(t). 
\end{eqnarray}}Observe that the same switching signal is used for the measured sensor outputs and the actuator inputs. Therefore, it 
is clear that if the state estimation is possible for a switching signal, then state feedback laws can be implemented for 
that particular switching signal. 
% One can also obtain the average energy required for a finite horizon LQR problem by considering those switching signals 
% for which state estimation is feasible. Note that the convergence to the origin for an LQR problem in this set up is not asymptotic. 
Suppose the observability matrix {\small$O_{\sigma(T)}(C,A)$} is full column rank for a particular signal {\small$\sigma$}. Hence, {\small$x(0)$} is uniquely determined. Thus, 
from the input-state equation {\small$x(t+1)=Ax(t)+B\sigma(t)u(t)$}, one can find the current state {\small$x(t)$} which allows us to implement state feedback laws. 
% One can also build an observer as follows: 
% \begin{eqnarray}
%  \hat{x}(t+1)=A\hat{x}(t)+B\sigma(t)u(t)+LC\sigma(t)(x-\hat{x})(t).\label{observ}
% \end{eqnarray}
% Let $e=x-\hat{x}$. Therefore, 
% %\begin{eqnarray}
%  $e(t+1)=(A-LC\sigma(t))e(t)$. Thus, choosing an appropriate $L$ for a particular $\sigma$, it is possible to build a state estimator. 
%  The choice of $L$ for a state estimator changes according to the switching signal $\sigma$.
% %\end{eqnarray}

One can consider the finite horizon LQR problem for each switching signal $\sigma$ as follows:
{\small\begin{eqnarray}
 \mbox{min }J_{\sigma}:=x^{*}(T)Q_fx(T)+\sum_{t=0}^{T-1}x^{*}(t)Qx(t)+u^{*}(t)Ru(t)\label{lqr}
\end{eqnarray}}where {\small$Q,Q_f\ge0,R>0$}. Suppose the initial condition {\small$x(0)$} is fixed. 
For a fixed {\small$\sigma$}, we can consider $(\ref{switchcomb})$ as a linear time varying system and consider the LQR problem for linear time varying systems by choosing $A(t)=A$ 
and {\small$B(t)=B\sigma(t)$}. By solving 
the difference Riccati equation with time varying coefficient {\small$B(t)$}, we can obtain a state feedback {\small$u(t)=-K(t)x(t)$} 
(for a fixed signal {\small$\sigma$}) as a solution of the LQR problem. From the solution 
of the difference Riccati equation, we can compute the optimal cost say {\small$J_{\sigma}$} for each {\small$\sigma$}. Thus, we can compute the average LQR cost 
{\small$J_{avg}=\sum_{\sigma \in \mathcal{S}_c^T}P(\sigma)J_{\sigma}$} 
% (by taking into 
% consideration the probability of occurence of $\sigma$ and the optimal LQR cost for that $\sigma$) 
by considering 
all switching signals {\small$\mathcal{S}_c^T$} (or {\small$\mathcal{S}_o^T$}) for which the controllability and the observability matrix becomes full rank.
%where $P(\sigma)$ is the probability of occurence of $\sigma$. 
The average LQR cost can also be used as a selection metric for optimal actuator placement problem for a fixed initial state. 
%One can handle the corresponding LQG problem as well in this set up.

In future, we wish to consider different models of switching signals for actuators 
and sensors instead of the model considered here.

\section{conclusion and future work}
% We showed that the controllability Gramian for systems of the form $(\ref{switch})$ can be expressed as a Hadamard product of two matrices. 
% The classical case of the controllability Gramian for linear systems turns out to be a special case. 
We found necessary and sufficient conditions on the admissible signals $\sigma$ (which models systems with a packet loss) such that the 
controllability Gramian $W_{\sigma(t)}$ is positive definite 
for a fixed $t$. This allowed us to obtain necessary and sufficient conditions for an arbitrary state transfer for our models.% which only considers the number of non-zero entries 
%in the switching signal to conclude about the invertibility of the controllability Gramian. 
We considered the analogous state estimation problem as well. We introduced a notion of average input/output energy 
%provided that an arbitrary state transfer is possible and state estimation is possible in time $T$. %For multi-input/multi-output systems, we gave a lower bound on the average energy. 
and defined a new probabilistic measure %by defining the average controllability Gramian 
which allowed us to have a new selection criterion for optimal actuator/sensor placement problem for single input/multi-input systems with packet 
dropouts. We stated how feedback laws and LQR problem can be considered for these models. 

In future, we wish to extend the results obtained for more general systems 
by relaxing a few assumptions made here. We wish to develop efficient algorithms/heuristics to solve the optimal actuator/sensor placement problem and 
extend the other classical controllability metrics to systems with packet dropouts using similar ideas. 
Moreover, we wish to analyze the performance of the probabilistic measure defined in this article by 
using the notion of tight frames used in \cite{shchat}; where we expect that the tight frames would lead to an optimal solution \cite{shchat}. 
Furthermore, we wish to study the optimal actuator placement problem subject to energy bounds considered in \cite{tzoum}. 
There are networks where the probability
that a packet dropout could be a function of system states such
as transmission rate. In future, we wish to consider such state-dependent switching signals 
to obtain trade-offs between the allowable packet loss probability and the 
transmission rate such that the system remains controllable. Moreover, we also wish to incorporate time delays in our switching signals. 
% Limitation is that one can obtain a control input in principle but we don't know how to compute it in practice for different random switching signals.
\section{Acknowledgement}
Author is thankful to Prof. N. Athanasopoulos and Prof. R. M. Jungers for many useful discussions and comments.

\end{document}